\documentclass[12pt]{article}

\usepackage[T1]{fontenc}
\usepackage{amsmath,amssymb,amsthm,mathtools,bm}
\usepackage{booktabs}
\usepackage{graphicx}
\usepackage{subcaption}
\usepackage[hidelinks]{hyperref}
\usepackage{url}
\usepackage{enumitem}
\usepackage{xcolor}
\usepackage{array}
\usepackage{tabularx}
\usepackage{microtype}
\usepackage{placeins}
\usepackage{geometry}

\newcolumntype{L}[1]{>{\raggedright\arraybackslash}p{#1}}
\newcolumntype{C}[1]{>{\centering\arraybackslash}p{#1}}

\newtheorem{theorem}{Theorem}[section]
\newtheorem{proposition}[theorem]{Proposition}

\newtheorem{definition}[theorem]{Definition}

\newcommand{\R}{\mathbb{R}}
\newcommand{\norm}[1]{\left\lVert #1 \right\rVert}

\title{A Two-Channel F-Transform Representation for Early Trajectory Characterization in Iterated Correlation Dynamics}
\author{
  Ishrak Alhajj Hassan \\
  \normalsize Department of Mathematics, Faculty of Science, University of Ostrava, Czech Republic \\
  \texttt{ishrak.alhajj.s01@osu.cz}
}
\date{}  

\usepackage{etoolbox}
\patchcmd{\maketitle}{\vskip 2em}{\vskip 1em}{}{}

\begin{document}

\maketitle

\begin{abstract}
\noindent
Many nonlinear iterative systems generate high-dimensional trajectories whose early behavior is informative but difficult to compare directly. This paper derives a fixed-dimensional F-transform coordinate representation for early trajectories of iterated Pearson correlation matrices. The construction is defined on the first five-point post-transient signal window, which is the shortest sampled window that simultaneously places the three symmetric fuzzy nodes at observed positions and supports a nondegenerate centered first-degree F-transform coefficient, thereby providing the earliest feasible local level--trend characterization within this sampled geometry.
The representation combines two logarithmic observables of the post-transient dynamics: step size, measuring update magnitude, and contraction ratio, measuring the local evolution of that magnitude. For each five-point signal, the three zero-degree F-transform coefficients encode the left, central, and right fuzzy-local levels, while the centered first-degree coefficient captures the local trend around the central node. Applying this same four-coordinate construction to the step-size and contraction-ratio signals yields the eight-dimensional descriptor
\(\Psi=(v_1,v_2,v_3,s_2,u_1,u_2,u_3,r_2)\),
with a common coordinate form across matrix sizes. For the fixed construction,
\(\Psi=M(q_2,\ldots,q_7)^{\top}\), \(q_k=\log\delta_k\), with
\(\operatorname{rank}M=6\). Thus, the descriptor is an injective,
overcomplete representation of the six logged step sizes underlying the
two channels, preserving the information in the observed segment while
reorganizing it into interpretable fuzzy-local levels and trends. The
representation is Lipschitz stable, and the centered first-degree
coefficient recovers affine trends exactly.

Convergence-length approximation is used as a downstream test of retained
dynamical information. Across 22 matrix dimensions and 22,000 trajectories,
repeated train--test evaluation shows predictive performance comparable to
raw two-channel and statistical-summary representations. Observation-matched
and sensitivity analyses indicate that performance depends more strongly on
the amount of trajectory information observed than on nearby choices in the
fuzzy construction. PCA shows that the first two principal components explain
on average \(84.47\%\) of the descriptor variance. Clustering reveals
reproducible coarse organization, with the strongest mean silhouette at
\(k=2\) and high stability for smaller numbers of clusters.

\end{abstract}

\noindent\textbf{Keywords:} F-transform \textbullet\ fuzzy feature extraction \textbullet\ trajectory representation \textbullet\ iterated correlation matrices \textbullet\ contraction dynamics \textbullet\  convergence characterization \textbullet\ soft computing

\section{Introduction}\label{sec:introduction}

Nonlinear iterative systems often generate trajectories whose early
behavior carries information about convergence behavior. In many computational settings, however, such trajectories are inconvenient objects for analysis: they may have variable stopping times, their raw states may be high-dimensional, and direct comparison of trajectory segments may obscure the local behavior that distinguishes one run from another. A useful representation should therefore convert a short trajectory segment into fixed-dimensional coordinates while retaining information relevant to convergence, comparison, visualization, and exploratory geometric analysis.

This paper studies such a representation problem for iterated Pearson correlation matrices. Given an initial matrix \(P_0\in\R^{n\times n}\), the iteration generates a trajectory
\[
P_{k+1}=f(P_k),\qquad k\ge 0.
\]
Under the row-wise Pearson correlation update, each row is centered, normalized to unit Euclidean norm, and correlated with all other normalized rows. From the first iterate onward, the trajectory lies in the set of symmetric positive semidefinite matrices with unit diagonal. This update is elementary to compute, but its global behavior is mathematically nontrivial. Classical work on repeated correlations and CONCOR reported rapid numerical stabilization and convergence toward block-pattern structures, particularly in psychometrics and social network blockmodeling \cite{McQuittyClark1968,BreigerBoormanArabie1975,KruskalCONCOR,Chen2002}.

Recent work has organized this iteration quantitatively. A large-scale empirical study reported stable laws across dimensions from 3 to 2000, including a distinguished first-step contraction, nearly monotone post-transient decay, a stable relationship between contraction ratio and step size, and empirically bounded convergence lengths under the tested initialization model \cite{EmpiricalLaws}. A subsequent finite-step analysis converted part of this empirical structure into probabilistic bounds for contraction ratios conditioned on normalized step size \cite{FiniteBounds}. These results describe important global and stepwise properties of the dynamics. They do not, however, provide a fixed coordinate representation of an early trajectory segment that can be used directly for comparison, visualization, or downstream statistical learning.

The present paper addresses this representation problem using the F-transform framework. The F-transform replaces a signal by local coefficients computed over a fuzzy partition \cite{Perfilieva2006,PerfilievaDankovaBede2011}. Zero-degree coefficients summarize local levels, while first-degree coefficients capture local linear trends \cite{KreinovichPerfilieva2011}. This value--trend separation has been useful in fuzzy modeling, time-series processing, filtering, and similarity analysis \cite{NovakPerfilievaHolcapekKreinovich2014,NovakMirshahi2021,InsightFuzzyModeling}. Here, the F-transform is used as a fuzzy feature-extraction mechanism for partial dynamical information.

The representation is built from two scalar observables of the matrix trajectory. The first is the Frobenius step size
\[
\delta_k=\|P_{k+1}-P_k\|_F,
\]
which measures the magnitude of one update. The second is the contraction ratio
\[
\rho_k=\frac{\delta_{k+1}}{\delta_k},
\]
which describes the local contraction behavior from one step to the next. On a short post-transient segment, the logarithms of these two signals are projected onto a three-node triangular fuzzy partition. Three zero-degree coefficients and one centered first-degree coefficient are retained for each channel, producing the eight-dimensional descriptor
\[
\Psi=(v_1,v_2,v_3,s_2,u_1,u_2,u_3,r_2).
\]
The first four coordinates describe local levels and trend of the step-size channel, while the last four provide the corresponding representation of the contraction-ratio channel.

A structural feature of this construction becomes visible once the two channels are written through the common logged step-size sequence
\[
q_k=\log \delta_k.
\]
Because
\[
\log\rho_k=q_{k+1}-q_k,
\]
the five step-size observations and five contraction-ratio observations used by the descriptor are generated by the six quantities
\[
q_2,\ldots,q_7.
\]
For the fixed triangular partition used here, the complete descriptor can therefore be written as
\[
\Psi=M(q_2,\ldots,q_7)^T
\]
for an explicit \(8\times6\) matrix \(M\). We show that
\[
\operatorname{rank}M=6.
\]
Hence the eight coordinates form an injective, overcomplete representation of the six underlying logged step sizes. This result gives an exact description of the information structure of the descriptor: the additional coordinates do not introduce additional independent observations, but reorganize the same early trajectory information into fuzzy-local levels and trends in the step-size and contraction-ratio channels.

Convergence-length approximation is used here as a downstream test of
the dynamical information retained by the representation, rather than
as the primary objective of the method.

The empirical question is then whether this structured coordinate
system retains useful information about convergence length while
providing a representation that is structured, interpretable, and
suitable for geometric analysis. We evaluate this question across 22 matrix dimensions and 1000 trajectories per dimension. The full descriptor is compared with raw step-size and contraction-ratio samples, statistical summaries, PCA-transformed raw coordinates, the step-size-only F-transform descriptor, and observation-matched controls. Random Forest regression is used as a common downstream benchmark for convergence-length approximation, with repeated train--test splits used to assess the stability of the comparison. Additional analyses examine sensitivity to the fuzzy construction, the position and length of the observed trajectory segment, convergence-target stability, and an equal-dimensional non-fuzzy local-summary representation.

Prediction is only one use of the coordinate system. We also study the geometry of the standardized descriptor using PCA and \(k\)-means clustering. The PCA analysis measures the extent to which variation in the eight coordinates concentrates along a small number of dominant directions. The clustering analysis examines coarse organization of the descriptor space over \(k=2,\ldots,8\), together with partition stability and a permuted-feature reference. These analyses provide complementary views of how early contraction trajectories are organized after the fuzzy-local transformation.

The contributions of the paper are as follows.
\begin{enumerate}[leftmargin=2em]

\item A two-channel F-transform descriptor is constructed for early trajectories of iterated correlation matrices. It separates local level from local trend in both the step-size and contraction-ratio channels.

\item The information structure of the descriptor is characterized exactly. The eight coordinates satisfy
\[
\Psi=M(q_2,\ldots,q_7)^T,
\qquad
\operatorname{rank}M=6,
\]
so the descriptor is an injective, overcomplete representation of the six logged step sizes underlying the two channels.

\item The representation is evaluated across 22 matrix dimensions and 1000 trajectories per dimension using a repeated-split Random Forest benchmark and controlled raw, statistical, PCA, observation-matched, and local-summary comparisons. Sensitivity analyses examine fuzzy-design choices, the observed
trajectory segment, and the convergence criterion.

\item PCA and clustering analyses characterize the geometry of the descriptor. The first two principal components explain on average \(84.47\%\) of the variance, while the clustering analysis identifies stable coarse organization with the strongest mean silhouette at \(k=2\).

\end{enumerate}

The paper is organized as follows. Section~\ref{sec:background} gives background on iterated correlation dynamics and F-transforms. Section~\ref{sec:representation} defines the two-channel representation and establishes its structural properties. Section~\ref{sec:experimental-framework} describes the experimental pipeline and compared feature families. Section~\ref{sec:results} presents the predictive, sensitivity, and geometric results. Section~\ref{sec:discussion} discusses interpretation, limitations, and future work.

\section{Background}\label{sec:background}

\subsection{Iterated Pearson correlation dynamics}

Let \(P_k\in\R^{n\times n}\). The Pearson correlation update maps
\(P_k\) to \(P_{k+1}\) by centering each row, normalizing each
centered row, and forming the Gram matrix of the resulting row
vectors. If \(Z_k\) denotes the row-centered and row-normalized
version of \(P_k\), then
\[
P_{k+1}=Z_k Z_k^T.
\]
Thus, from the first update onward, \(P_k\) is symmetric positive
semidefinite with diagonal entries equal to one. The trajectory
therefore evolves in the elliptope of correlation matrices.

Repeated-correlation procedures have a long history in psychometrics,
network analysis, and information visualization. Iterative
intercolumnar correlation was studied by McQuitty and Clark
\cite{McQuittyClark1968}, while CONCOR-type blockmodeling
\cite{BreigerBoormanArabie1975,KruskalCONCOR} and generalized
association plots \cite{Chen2002} provide further examples of
iterated-correlation procedures. These procedures motivate the study
of convergence behavior, while the present paper focuses on a
specific representation problem: how to summarize the early
contraction behavior of such iterations in fixed, interpretable
coordinates.

More recent time-series representation methods learn latent coordinates
directly from data. TS2Vec constructs hierarchical contrastive
representations across temporal scales \cite{Yue2022}, while T-Rep
incorporates explicit temporal information into learned representations
\cite{Fraikin2024}. In fuzzy learning, locality-preserving projected
fuzzy clustering provides another example of combining fuzzy structure
with learned low-dimensional coordinates \cite{Zhou2021}. The present
construction differs from these approaches in that the representation is
fixed analytically rather than learned from data, and every coordinate
has an explicit dynamical interpretation.
The objective here is therefore different from learning an
accuracy-optimized latent embedding: the representation is fixed
independently of the downstream prediction target, its information
structure can be characterized exactly, and every coordinate has an
explicit dynamical interpretation. For this reason, the principal
empirical controls emphasize raw, statistical, PCA-based, and
observation-matched representations that allow the effect of coordinate
organization to be isolated directly.

\subsection{F-transform and fuzzy local summaries}

Let \([a,b]\subset\R\). A family
\(A=\{A_1,\dots,A_m\}\), with
\(A_j:[a,b]\to[0,1]\), is a fuzzy partition with nodes \(c_j\)
if each \(A_j\) is localized around \(c_j\), satisfies
\(A_j(c_j)=1\), and fulfills the Ruspini condition
\[
\sum_{j=1}^{m}A_j(t)=1,\qquad t\in[a,b].
\]
For a bounded signal \(z:[a,b]\to\R\), the zero-degree F-transform
coefficients are weighted local averages:
\begin{equation}\label{eq:f0-continuous}
F_j^0[z]
=
\frac{\int_a^b z(t)A_j(t)\,dt}
     {\int_a^b A_j(t)\,dt},
\qquad j=1,\dots,m.
\end{equation}
The first-degree F-transform gives a local linear approximation
\begin{equation}\label{eq:f1-local}
F_j^1[z](t)
=
\beta_j^0+\beta_j^1(t-c_j),
\end{equation}
where \(\beta_j^0=F_j^0[z]\), and
\begin{equation}\label{eq:f1-continuous}
\beta_j^1
=
\frac{\int_a^b z(t)(t-c_j)A_j(t)\,dt}
     {\int_a^b (t-c_j)^2A_j(t)\,dt}.
\end{equation}
For discrete signals, the integrals are replaced by finite sums.

In this work, the F-transform is used as a feature-extraction
mechanism. Zero-degree coefficients encode local levels, whereas a
centered first-degree coefficient encodes local trend. Because
neighboring fuzzy supports overlap, intermediate observations
contribute to adjacent local summaries rather than being assigned to
disjoint temporal blocks. This gives the resulting coordinates a
direct local level--trend interpretation.

\section{Two-Channel Early Trajectory Representation}
\label{sec:representation}

\subsection{Contraction observables}

For a trajectory \((P_k)\), define the Frobenius step size
\begin{equation}\label{eq:delta}
\delta_k
=
\norm{P_{k+1}-P_k}_F,
\end{equation}
where
\[
\norm{M}_F
=
\left(
\sum_{i=1}^n\sum_{j=1}^n M_{ij}^2
\right)^{1/2}.
\]
Whenever \(\delta_k>0\), define the contraction ratio
\begin{equation}\label{eq:rho}
\rho_k
=
\frac{\delta_{k+1}}{\delta_k}.
\end{equation}
The two sequences describe different aspects of the same contraction
trajectory. The step-size sequence measures the magnitude of each
update, whereas the contraction-ratio sequence expresses how that
magnitude changes from one step to the next.

Values \(\rho_k>1\) are retained and represent local increases in
step size (overshoots), whereas \(\rho_k<1\) represents local
contraction; no monotonic decrease of \(\delta_k\) is assumed.
Such local overshoots are consistent with the nearly monotone
post-transient behavior reported previously for this iteration
\cite{EmpiricalLaws}.

The response variable used in the predictive experiments is the
convergence length
\begin{equation}\label{eq:nconv}
N_{\mathrm{conv}}
=
\min\{k\ge 1:\delta_k<10^{-6}\}.
\end{equation}

\subsection{Early post-transient segment}

The main descriptor uses the five post-transient positions
\begin{equation}\label{eq:window}
T=\{2,3,4,5,6\}.
\end{equation}
The segment begins after the distinguished initial contraction
observed in previous empirical studies of the iteration
\cite{EmpiricalLaws}.

For positive step sizes, write
\begin{equation}\label{eq:q-def}
q_k=\log\delta_k.
\end{equation}
The logarithmic step-size channel is then
\begin{equation}\label{eq:x-signal}
x
=
(q_2,q_3,q_4,q_5,q_6),
\end{equation}
whereas the logarithmic contraction-ratio channel is
\begin{equation}\label{eq:y-signal}
y
=
(\log\rho_2,\log\rho_3,\log\rho_4,\log\rho_5,\log\rho_6).
\end{equation}
Since
\begin{equation}\label{eq:logrho}
\log\rho_k
=
q_{k+1}-q_k,
\end{equation}
the second channel can equivalently be written as
\begin{equation}\label{eq:y-q}
y
=
(q_3-q_2,\,
 q_4-q_3,\,
 q_5-q_4,\,
 q_6-q_5,\,
 q_7-q_6).
\end{equation}

Thus both channels contain five values, but the complete two-channel
construction is determined by the six consecutive logged step sizes
\[
q_2,\ldots,q_7.
\]
In particular, \(q_7\) enters intrinsically through the final
contraction ratio
\[
\rho_6=\frac{\delta_7}{\delta_6}.
\]
This six-step segment is therefore the underlying observed information
used by the full descriptor.

\subsection{Fuzzy projection}

Identify the five positions in \(T\) with the local coordinates
\(\{0,1,2,3,4\}\). We use three triangular basic functions with
centers
\[
c_1=0,\qquad c_2=2,\qquad c_3=4,
\]
and bandwidth \(h=2\):
\begin{equation}\label{eq:partition}
A_j(t)
=
\max\left(
0,
1-\frac{|t-c_j|}{h}
\right),
\qquad j=1,2,3.
\end{equation}
On the discrete five-point segment, their weights are
\[
A_1=(1,\tfrac12,0,0,0),
\qquad
A_2=(0,\tfrac12,1,\tfrac12,0),
\qquad
A_3=(0,0,0,\tfrac12,1).
\]
They satisfy the Ruspini condition at each sampled position.

For a discrete signal
\(z=(z_0,z_1,z_2,z_3,z_4)\), define
\begin{equation}\label{eq:f0-discrete}
F_j^0[z]
=
\frac{\sum_{t=0}^4 z_tA_j(t)}
     {\sum_{t=0}^4 A_j(t)},
\qquad j=1,2,3,
\end{equation}
and retain the centered first-degree coefficient
\begin{equation}\label{eq:f1-discrete}
\beta_2^1[z]
=
\frac{\sum_{t=0}^4
z_t(t-c_2)A_2(t)}
{\sum_{t=0}^4
(t-c_2)^2A_2(t)}.
\end{equation}
For the present partition these coefficients have the explicit form
\begin{equation}\label{eq:closed-form-f}
\begin{aligned}
F_1^0[z]&=\frac{2z_0+z_1}{3},\\
F_2^0[z]&=\frac{z_1+2z_2+z_3}{4},\\
F_3^0[z]&=\frac{z_3+2z_4}{3},\\
\beta_2^1[z]&=\frac{z_3-z_1}{2}.
\end{aligned}
\end{equation}

Applying these coefficients to the step-size and contraction-ratio
channels gives
\[
v_j=F_j^0[x],
\qquad
u_j=F_j^0[y],
\qquad
s_2=\beta_2^1[x],
\qquad
r_2=\beta_2^1[y].
\]

\begin{definition}[Two-channel early trajectory representation]
The two-channel F-transform representation of a trajectory \(P\) is
\begin{equation}\label{eq:representation}
\Psi(P)
=
(v_1,v_2,v_3,s_2,u_1,u_2,u_3,r_2)
\in\R^8.
\end{equation}
Its step-size component is
\[
\Psi_\delta(P)
=
(v_1,v_2,v_3,s_2),
\]
and its contraction-ratio component is
\[
\Psi_\rho(P)
=
(u_1,u_2,u_3,r_2).
\]
\end{definition}

The eight coordinates therefore provide overlapping fuzzy-local
levels and centered trends for the two dynamical observables. They are
not eight independent observations; their exact information structure
is characterized next.

\subsection{Exact structure and injectivity}

Let
\[
q
=
(q_2,q_3,q_4,q_5,q_6,q_7)^T.
\]
Using \eqref{eq:closed-form-f} together with
\eqref{eq:y-q}, the descriptor is a linear transformation of \(q\).

\begin{theorem}[Exact structure of the descriptor]
\label{thm:rank-structure}
For the three-node triangular construction,
\begin{equation}\label{eq:matrix-map}
\Psi=Mq,
\end{equation}
where
\begin{equation}\label{eq:M-matrix}
M=
\begin{pmatrix}
 \frac23 & \frac13 & 0 & 0 & 0 & 0 \\[2pt]
 0 & \frac14 & \frac12 & \frac14 & 0 & 0 \\[2pt]
 0 & 0 & 0 & \frac13 & \frac23 & 0 \\[2pt]
 0 & -\frac12 & 0 & \frac12 & 0 & 0 \\[2pt]
 -\frac23 & \frac13 & \frac13 & 0 & 0 & 0 \\[2pt]
 0 & -\frac14 & -\frac14 & \frac14 & \frac14 & 0 \\[2pt]
 0 & 0 & 0 & -\frac13 & -\frac13 & \frac23 \\[2pt]
 0 & \frac12 & -\frac12 & -\frac12 & \frac12 & 0
\end{pmatrix}.
\end{equation}
Moreover,
\begin{equation}\label{eq:rankM}
\operatorname{rank}M=6.
\end{equation}
Consequently, the map
\[
q\longmapsto\Psi
\]
is injective.
\end{theorem}

\begin{proof}
The rows of \(M\) follow directly from the explicit F-transform
coefficients in \eqref{eq:closed-form-f}. For example,
\[
v_1=\frac{2q_2+q_3}{3},
\qquad
s_2=\frac{q_5-q_3}{2},
\]
while
\[
u_1
=
\frac{2(q_3-q_2)+(q_4-q_3)}{3}
=
\frac{-2q_2+q_3+q_4}{3}.
\]
The remaining rows are obtained similarly.

To establish the rank, consider the \(6\times6\) minor formed by
rows \(1,2,3,4,5,7\) of \(M\). Its determinant is
\[
\frac{2}{81}\neq0.
\]
Hence \(\operatorname{rank}M\ge6\). Since \(M\) has six columns,
\(\operatorname{rank}M=6\). Therefore its null space is trivial and
the map \(q\mapsto\Psi\) is injective.
\end{proof}

The eight coordinates are therefore an overcomplete representation of
six underlying logged step sizes. In particular, they satisfy two
independent linear identities:
\begin{align}
-3v_1+6v_2-3v_3-5s_2-3u_1+8u_2 &=0,
\label{eq:dependency1}\\
-3v_1+6v_2-3v_3+3s_2-3u_1+4r_2 &=0.
\label{eq:dependency2}
\end{align}
These identities make the redundancy explicit. At the same time,
Theorem~\ref{thm:rank-structure} shows that no information contained
in \(q_2,\ldots,q_7\) is lost by passing to the eight F-transform
coordinates.

The role of the contraction-ratio channel is therefore structural
rather than informationally independent. Because
\(\log\rho_k=q_{k+1}-q_k\), it does not introduce new underlying
trajectory observations beyond \(q_2,\ldots,q_7\). Its purpose is to
make local contraction behavior explicit in the coordinate system,
alongside the local levels and trend of the step-size channel.

\subsection{Stability and trend interpretation}

The descriptor also has a direct stability property.

\begin{proposition}[Lipschitz stability]
\label{prop:stability}
For any two pairs of five-point signals
\((x,y)\) and \((\tilde x,\tilde y)\),
\[
\|\Psi(x,y)-\Psi(\tilde x,\tilde y)\|_2
\le
2\left(
\|x-\tilde x\|_\infty
+
\|y-\tilde y\|_\infty
\right).
\]
\end{proposition}

\begin{proof}
Each zero-degree coefficient in
\eqref{eq:closed-form-f} is a normalized nonnegative weighted average,
and therefore
\[
|F_j^0[z]-F_j^0[\tilde z]|
\le
\|z-\tilde z\|_\infty.
\]
For the centered slope coefficient,
\[
\left|
\beta_2^1[z]-\beta_2^1[\tilde z]
\right|
=
\frac12
\left|
(z_3-\tilde z_3)
-
(z_1-\tilde z_1)
\right|
\le
\|z-\tilde z\|_\infty.
\]
Thus each of the four coordinates associated with one channel is
bounded by the sup-norm perturbation of that channel. Consequently,
\[
\|\Psi(x,y)-\Psi(\tilde x,\tilde y)\|_2^2
\le
4\|x-\tilde x\|_\infty^2
+
4\|y-\tilde y\|_\infty^2,
\]
which gives the stated bound.
\end{proof}

The centered first-degree coefficient also recovers affine trends
exactly.

\begin{proposition}[Exact affine-trend recovery]
\label{prop:affine-trend}
If
\[
z_t=a+bt,
\qquad t=0,\ldots,4,
\]
then
\[
\beta_2^1[z]=b.
\]
\end{proposition}

\begin{proof}
From \eqref{eq:closed-form-f},
\[
\beta_2^1[z]
=
\frac{z_3-z_1}{2}
=
\frac{(a+3b)-(a+b)}{2}
=
b.
\]
\end{proof}

The zero-degree coordinates therefore describe fuzzy-local levels,
while the centered first-degree coordinates recover the local slope
exactly for affine signals. Together with the rank characterization,
these properties give the descriptor a precise interpretation as a
stable, injective, fuzzy-local coordinate system for the observed
early contraction segment.

\section{Experimental Framework}
\label{sec:experimental-framework}

\subsection{Dataset generation and trajectory recording}

The experiments were conducted for
\[
\begin{aligned}
n\in\{&3,6,7,9,12,16,23,30,40,50,55,69,80,100,150,\\
&250,350,600,1054,1600,1700,2000\}.
\end{aligned}
\]
For each dimension, 1000 initial matrices were generated with entries
independently sampled from the uniform distribution on \([-1,1]\).
The same deterministic seeding convention was used throughout the
study to make the trajectory generation reproducible.

For every trajectory, the iterated Pearson correlation map was applied
and the Frobenius step sizes
\[
\delta_k=\|P_{k+1}-P_k\|_F
\]
were recorded. The primary convergence target remained the first
crossing of the tolerance \(10^{-6}\),
\begin{equation}\label{eq:experimental-target}
N_{\mathrm{conv}}
=
\min\{k\ge1:\delta_k<10^{-6}\}.
\end{equation}
Trajectory generation did not terminate at this first crossing.
Instead, genuine step sizes were retained through the early diagnostic
segment and the iteration was continued beyond the first threshold
crossing so that recrossings, sustained convergence, and nearby
tolerance levels could be evaluated. No uncomputed post-convergence
values were replaced by repeated or padded values.

The original operational limit of 100 iterations was retained as a
diagnostic reference, while the trajectory generator allowed
continuation up to 120 iterations when needed. Genuine values
\(\delta_1,\ldots,\delta_{12}\) were stored for the early-trajectory
analyses. After the first \(10^{-6}\) crossing, at least ten additional
updates were computed, subject to the diagnostic stopping conditions,
so that immediate threshold recrossings could be detected.

For numerical construction of the logarithmic signals, a common floor
\(\varepsilon=10^{-12}\) was used:
\begin{equation}\label{eq:q-epsilon}
\delta_k^{(\varepsilon)}
=
\max\{\delta_k,\varepsilon\},
\qquad
q_k
=
\log\delta_k^{(\varepsilon)}.
\end{equation}
The contraction-ratio channel was then formed from this same sequence,
\begin{equation}\label{eq:ratio-numerical}
\log\rho_k^{(\varepsilon)}
=
q_{k+1}-q_k.
\end{equation}
Using one common logged step-size sequence ensures numerical
consistency between the two channels.
Because both numerical channels are constructed from the same
\(q_k\) sequence, the linear representation
\(\Psi=M(q_2,\ldots,q_7)^T\) and the exact linear dependencies
established in Theorem~\ref{thm:rank-structure} apply unchanged to
the numerically implemented descriptor.

\subsection{Compared feature families}

The main experiments compare the proposed representation with raw,
statistical, and transformed representations of the same early
trajectory information.

The principal feature families are as follows.

\begin{enumerate}[leftmargin=2em]

\item \textbf{Step-size F-transform.}
The four coordinates
\[
(v_1,v_2,v_3,s_2)
\]
computed from
\[
q_2,\ldots,q_6.
\]

\item \textbf{Full two-channel F-transform.}
The eight coordinates
\[
(v_1,v_2,v_3,s_2,u_1,u_2,u_3,r_2)
\]
defined in Section~\ref{sec:representation}. The complete
construction uses the underlying segment
\[
q_2,\ldots,q_7.
\]

\item \textbf{Raw contraction-ratio samples.}
The five values
\[
q_3-q_2,\,
q_4-q_3,\,
q_5-q_4,\,
q_6-q_5,\,
q_7-q_6.
\]

\item \textbf{Raw two-channel representation.}
The five step-size values
\[
q_2,\ldots,q_6
\]
together with the five logarithmic contraction-ratio values above,
giving ten model inputs. As established in
Theorem~\ref{thm:rank-structure}, these ten channel values are
generated by the same six underlying quantities
\(q_2,\ldots,q_7\).

\item \textbf{Statistical summaries.}
For each of the two five-point channels, the mean, standard deviation,
minimum, maximum, range, linear slope, and intercept were computed,
giving fourteen model inputs.

\item \textbf{PCA-transformed raw representation.}
The ten raw two-channel coordinates were standardized and transformed
by PCA, with eight components supplied to the regression model. This
provides a nominally eight-dimensional transformed-coordinate
baseline for comparison with the eight F-transform coordinates.
Because the raw channels themselves are generated by six underlying
logged step sizes, retaining eight PCA coordinates does not imply
eight independent directions.

\end{enumerate}

Two additional controls were included to separate the effect of the
representation from the amount of trajectory information observed.

The first is the six-dimensional raw matched-segment representation
\[
(q_2,q_3,q_4,q_5,q_6,q_7),
\]
which exposes directly the same six underlying observations used by
the full F-transform descriptor.

The second augments the four-dimensional step-size F-transform by the
single additional value \(q_7\),
\[
(v_1,v_2,v_3,s_2,q_7).
\]
This control distinguishes the effect of observing the additional step
from the effect of organizing the same information through the
contraction-ratio channel.

The representation families and their underlying observed segments
are summarized in Table~\ref{tab:families}.

\begin{table}[ht]
\centering
\caption{Principal representation families used in the controlled benchmark.}
\label{tab:families}
\small
\begin{tabularx}{\textwidth}{L{0.29\textwidth} C{0.10\textwidth} C{0.21\textwidth} >{\raggedright\arraybackslash}X}
\toprule
Representation & Inputs & Observed segment & Role \\
\midrule
Step F-transform & 4 & \(q_2,\ldots,q_6\) & Step-size-channel fuzzy-local representation \\
Raw ratio only & 5 & \(q_2,\ldots,q_7\) & Direct contraction-ratio baseline \\
Step F-transform \(+q_7\) & 5 & \(q_2,\ldots,q_7\) & Observation-matched control for the step-size channel \\
Raw matched segment & 6 & \(q_2,\ldots,q_7\) & Direct coordinates of the same six underlying observations \\
Full F-transform & 8 & \(q_2,\ldots,q_7\) & Proposed structured fuzzy-local representation \\
PCA raw full & 8 & \(q_2,\ldots,q_7\) & Eight-component linear transformed-coordinate baseline \\
Raw step/ratio & 10 & \(q_2,\ldots,q_7\) & Unstructured two-channel coordinate baseline \\
Statistical full summary & 14 & \(q_2,\ldots,q_7\) & Engineered global summary baseline \\
\bottomrule
\end{tabularx}
\end{table}

\subsection{Prediction protocol and statistical comparisons}

Random Forest regression was used as a common downstream benchmark for
measuring how much information each representation retains about
\(N_{\mathrm{conv}}\). The purpose of this experiment is representation
comparison rather than optimization of a particular predictive model.

All models were fitted separately within each matrix dimension. For
each dimension, ten random 80/20 train--test partitions were generated
using split seeds
\[
0,1,\ldots,9.
\]
The same partition was used for every feature family at a given split.
Each Random Forest contained 300 trees, and the estimator random state
was fixed at 42 throughout the benchmark. Thus, variation across the
ten repetitions reflects variation in the train--test partition rather
than a simultaneous change in the fitted estimator's random seed.

Predictive performance was evaluated using the coefficient of
determination \(R^2\) and root mean squared error (RMSE). For each
matrix dimension and feature family, the ten split-level scores were
first averaged. Overall reported values were then obtained by
averaging these 22 dimension-level means, so that each tested matrix
dimension contributes equally to the aggregate result.

For the PCA predictive baseline, standardization and PCA were fitted
within each training set and then applied to the corresponding test
set before Random Forest regression. This prevents information from
the held-out observations from entering the fitted transformation.

Dimension-level mean \(R^2\) values were also compared pairwise. For a
pair of representations \(A\) and \(B\), the difference
\[
d_n
=
\overline{R^2}_{A,n}
-
\overline{R^2}_{B,n}
\]
was computed for each of the 22 dimensions. The analysis reports the
mean paired difference, a two-sided 95\% \(t\)-interval for the mean
difference, dimension-level win--loss counts, and a two-sided Wilcoxon
signed-rank test.

Six principal comparisons were evaluated:
the full F-transform against the raw matched segment, the
step-size F-transform augmented by \(q_7\) against the step-size
F-transform, the full F-transform against the \(q_7\)-augmented
step-size descriptor, and the full F-transform against the raw
two-channel, statistical-summary, and PCA baselines. The corresponding
Wilcoxon \(p\)-values were adjusted jointly using the Holm procedure.
These intervals and tests summarize the consistency of representation
differences across the 22 tested matrix dimensions; they are not
intended as inference to a randomly sampled population of matrix
dimensions.

Random Forest impurity-based feature importance was examined
separately for four feature families: the full F-transform, raw
two-channel representation, statistical summary, and step-size
F-transform. Feature importances were averaged first across the ten
splits within each dimension and then across the 22 dimensions.
Two-sided 95\% \(t\)-intervals across dimension-level means were used
to show variation in the importance estimates. Because the
F-transform coordinates satisfy exact linear dependencies,
Theorem~\ref{thm:rank-structure}, these values are interpreted as
model-dependent importance measures rather than as an intrinsic
mathematical ranking of the coordinates.

\subsection{Sensitivity and diagnostic analyses}

The baseline fuzzy construction uses a five-point segment, three
triangular nodes, bandwidth \(h=2\), three zero-degree local levels,
and one centered first-degree coefficient per channel. Several
controlled alternatives were evaluated to determine how strongly the
predictive benchmark depends on these design choices.

The number of fuzzy nodes was varied from the baseline three-node
construction to one and five nodes. Within the three-node construction,
triangular membership weights with nominal bandwidths \(h=1.5\),
\(2\), and \(2.5\) were compared, together with Gaussian and boxcar
alternatives. For these sensitivity variants, the sampled membership
weights were normalized pointwise across nodes before the local
coefficients were computed. In addition, a zero-degree-only
construction was evaluated to measure the effect of removing the
centered first-degree trend coefficient.

Sensitivity to the amount of early trajectory information was studied
using channel lengths
\[
L\in\{4,5,6\}.
\]
An \(L\)-point two-channel construction is generated by \(L+1\)
consecutive logged step sizes because the final contraction ratio
requires the next step-size value. Accordingly, the baseline \(L=5\)
construction uses
\[
q_2,\ldots,q_7,
\]
while the shorter and longer versions use, respectively,
\[
q_2,\ldots,q_6
\qquad\text{and}\qquad
q_2,\ldots,q_8.
\]
All sensitivity configurations were evaluated under the same ten-split
Random Forest protocol as the principal benchmark. Each alternative
was compared with the baseline construction using dimension-level
paired \(R^2\) differences and a Wilcoxon signed-rank test, with Holm adjustment across the full family of sensitivity
comparisons.

To distinguish fuzzy overlap from local aggregation itself, an
equal-dimensional non-fuzzy local-summary representation was also
constructed. For one five-point channel
\[
z=(z_0,z_1,z_2,z_3,z_4),
\]
the four coordinates are
\begin{equation}\label{eq:hard-local}
L_z=\frac{z_0+z_1}{2},
\qquad
C_z=\frac{z_1+z_2+z_3}{3},
\qquad
R_z=\frac{z_3+z_4}{2},
\qquad
S_z=\frac{z_3-z_1}{2}.
\end{equation}
Applying the same construction to both the step-size and
contraction-ratio channels gives eight inputs on the same underlying
segment \(q_2,\ldots,q_7\). This provides a direct comparison between
overlapping fuzzy-local weighting and a hard local-summary
construction with the same model-input count.

The convergence target itself was checked separately. First-crossing
times were recorded at tolerances
\[
10^{-5},\quad10^{-6},\quad10^{-7},\quad10^{-8},
\]
and compared with the first index at which three consecutive step
sizes remained below the same tolerance. At the primary tolerance
\(10^{-6}\), the trajectory records were also examined for any return
above threshold after the first crossing. Finally, the analysis
recorded whether the complete segment \(q_2,\ldots,q_7\) lay strictly
before the first \(10^{-6}\) crossing and whether either the original
100-step limit or the extended diagnostic limit was reached.

\subsection{Geometric organization}

The geometric analysis uses the standardized eight-dimensional
F-transform descriptor and is distinct from the PCA predictive
baseline described above. Standardization was performed separately
within each matrix dimension.

For PCA, all eight descriptor coordinates were retained when computing
the variance decomposition. The first two principal coordinates were
used for visualization. The dimension \(n=350\) was used for the representative visualization.
This choice is illustrative only; all quantitative PCA summaries
reported in the text are computed across the complete set of 22
matrix dimensions, and no inference depends on this representative
plot.

Clustering was performed directly in the standardized
eight-dimensional descriptor space rather than in the two-dimensional
PCA projection. \(K\)-means partitions were examined for
\[
k=2,3,\ldots,8.
\]
For the principal silhouette calculation, \(K\)-means used random
state 42 and 50 initializations. Cluster quality was measured by the
silhouette coefficient.

Partition reproducibility was evaluated by repeating \(K\)-means over
ten initialization seeds and computing pairwise adjusted Rand indices
between the resulting partitions. A permuted-feature reference was
constructed by independently permuting each standardized descriptor
coordinate, thereby preserving the marginal distribution of every
coordinate while disrupting their multivariate organization. Ten such
reference realizations were generated for each dimension and value of
\(k\).

Two additional checks were performed. First, the observed separation
of cluster mean convergence lengths was compared with a reference
obtained by permuting \(N_{\mathrm{conv}}\) while keeping the cluster
assignments fixed. Second, the descriptor coordinates were perturbed
with small Gaussian noise in standardized units at standard deviations
\[
0.01,\qquad0.05,\qquad0.10,
\]
and adjusted Rand indices relative to the unperturbed partitions were
used to quantify clustering robustness.

Together, the predictive, sensitivity, convergence, PCA, and
clustering analyses evaluate complementary properties of the same
early-trajectory coordinate system: retained convergence information,
dependence on design choices, validity of the convergence target, and
geometric organization.

\section{Results}\label{sec:results}

\subsection{Predictive information retained by the descriptor}
\label{sec:predictive-results}

Table~\ref{tab:overall} summarizes predictive performance for the
principal representation families under the ten-split protocol. The
fourteen-dimensional statistical summary and the ten-dimensional raw
two-channel representation give the highest aggregate mean \(R^2\),
with values of \(0.6535\) and \(0.6534\), respectively. The proposed
eight-dimensional F-transform descriptor follows closely with mean
\(R^2=0.6491\). The corresponding mean values for the
PCA-transformed raw representation and the four-dimensional
step-size F-transform are \(0.6205\) and \(0.5091\), respectively.

\begin{table}[ht]
\centering
\caption{Ten-split convergence-length benchmark for the principal feature representations. Each row aggregates 22 dimension-level means.}
\label{tab:overall}
\small
\begin{tabular}{lrrrr}
\toprule
Representation & Inputs & Mean \(R^2\) & Median \(R^2\) & Mean RMSE \\
\midrule
Statistical full summary & 14 & 0.6535 & 0.7703 & 0.9058 \\
Raw step/ratio & 10 & 0.6534 & 0.7688 & 0.9068 \\
Full F-transform & 8 & 0.6491 & 0.7553 & 0.9148 \\
PCA raw full & 8 & 0.6205 & 0.7355 & 0.9897 \\
Raw ratio only & 5 & 0.5791 & 0.6371 & 1.0436 \\
Step F-transform & 4 & 0.5091 & 0.6143 & 1.1119 \\
\bottomrule
\end{tabular}
\end{table}

The principal result is therefore not an accuracy advantage over the
larger raw or statistical feature sets. Rather, the eight-dimensional
F-transform representation retains nearly the same
convergence-related predictive information while expressing the
observed trajectory through explicit fuzzy-local level and trend
coordinates. Its aggregate differences from the raw two-channel and
statistical-summary representations are small:
\[
0.6491-0.6534=-0.0043,
\qquad
0.6491-0.6535=-0.0044.
\]
The corresponding dimension-level paired comparisons do not indicate
a systematic difference between these representations after
adjustment for the family of principal comparisons.

The comparison with the four-dimensional step-size F-transform is
substantially larger. The full two-channel construction raises the
mean \(R^2\) from \(0.5091\) to \(0.6491\). This comparison alone,
however, combines two effects: the full descriptor also uses
\(q_7\), which is required to form
\[
\rho_6=\frac{\delta_7}{\delta_6}.
\]
The observation-matched controls separate the effect of this
additional trajectory observation from the effect of its structured
two-channel representation.

\begin{table}[ht]
\centering
\caption{Observation-matched controls for the two-channel representation. Each row aggregates 22 dimension-level means over ten train--test splits.}
\label{tab:matchedcontrols}
\small
\begin{tabular}{lrrr}
\toprule
Representation & Inputs & Mean \(R^2\) & Mean RMSE \\
\midrule
Full F-transform & 8 & 0.6491 & 0.9148 \\
Raw matched segment & 6 & 0.6421 & 0.9257 \\
Step F-transform \(+q_7\) & 5 & 0.6402 & 0.9270 \\
Step F-transform & 4 & 0.5091 & 1.1119 \\
\bottomrule
\end{tabular}
\end{table}

Augmenting the step-size F-transform by \(q_7\) increases its mean
\(R^2\) from \(0.5091\) to \(0.6402\). The mean dimension-level
increase is \(0.1311\), occurs in all 22 tested dimensions, and
remains significant after Holm adjustment. A substantial part of the
difference between the four-coordinate step-size descriptor and the
full representation is therefore associated with observing the
additional trajectory value required by the complete two-channel
construction.

Once the observed segment is matched, the full F-transform descriptor
has mean \(R^2=0.6491\), compared with \(0.6421\) for the six raw
values \(q_2,\ldots,q_7\). The mean paired difference is \(0.0071\)
in favor of the F-transform, with higher dimension-level mean
\(R^2\) in 16 of the 22 dimensions. The corresponding 95\% confidence
interval includes zero, however, and the paired Wilcoxon comparison
is not significant after Holm adjustment. The two representations
therefore provide comparable predictive access to the same observed
trajectory information.

Similarly, the full descriptor exceeds the step-size F-transform
augmented by \(q_7\) by a mean \(R^2\) difference of \(0.0090\).
The full descriptor is higher in 13 dimensions and the augmented
step-size descriptor in 9, while the paired comparison is not
significant after adjustment. The contraction-ratio channel is
therefore best interpreted as an explicit representation of local
contraction behavior rather than as a source of additional independent
trajectory observations.

Table~\ref{tab:paired} gives the complete paired dimension-level
comparisons.

\begin{table}[ht]
\centering
\caption{Paired comparison of dimension-level mean \(R^2\) values. Differences are model A minus model B. Holm adjustment covers the six listed comparisons.}
\label{tab:paired}
\scriptsize
\resizebox{0.98\textwidth}{!}{%
\begin{tabular}{llccccc}
\toprule
Model A & Model B & Mean diff. & 95\% CI & Wins--losses & Raw \(p\) & Holm \(p\) \\
\midrule
Full F-transform & Raw matched segment & 0.0071 & \([-0.0017,0.0159]\) & 16--6 & 0.1207 & 0.3936 \\
Step F \(+q_7\) & Step F-transform & 0.1311 & \([0.1015,0.1607]\) & 22--0 & $4.77\times 10^{-7}$ & $2.86\times 10^{-6}$ \\
Full F-transform & Step F \(+q_7\) & 0.0090 & \([-0.0011,0.0190]\) & 13--9 & 0.0984 & 0.3936 \\
Full F-transform & Raw step/ratio & -0.0043 & \([-0.0104,0.0019]\) & 7--15 & 0.1375 & 0.3936 \\
Full F-transform & Statistical full summary & -0.0044 & \([-0.0102,0.0014]\) & 7--15 & 0.1207 & 0.3936 \\
Full F-transform & PCA raw full & 0.0286 & \([0.0085,0.0487]\) & 17--5 & 0.0083 & 0.0416 \\
\bottomrule
\end{tabular}%
}
\end{table}

Figure~\ref{fig:performance-size} shows predictive performance across
the tested matrix dimensions for the six principal representation
families. Performance varies with dimension for all representations,
but the full F-transform remains close to the raw two-channel and
statistical-summary baselines over much of the tested range. The
step-size-only F-transform is consistently weaker across most
dimensions.

The raw contraction-ratio representation also retains substantial
predictive information by itself, showing that the local contraction
behavior described by \(\rho_k\) is strongly associated with
convergence length. This should be distinguished from
statistical independence: the contraction-ratio values are
algebraically determined by consecutive logged step sizes, as shown
in Section~\ref{sec:representation}.

The logarithmic horizontal axis in
Figure~\ref{fig:performance-size} is used only to display the tested
dimensions clearly because they are concentrated in the lower part of
the range while extending to \(n=2000\). All aggregate numerical
results give equal weight to each of the 22 matrix dimensions.

\begin{figure}[!htbp]
\centering

\begin{subfigure}[t]{0.48\textwidth}
\centering
\includegraphics[width=\linewidth]
{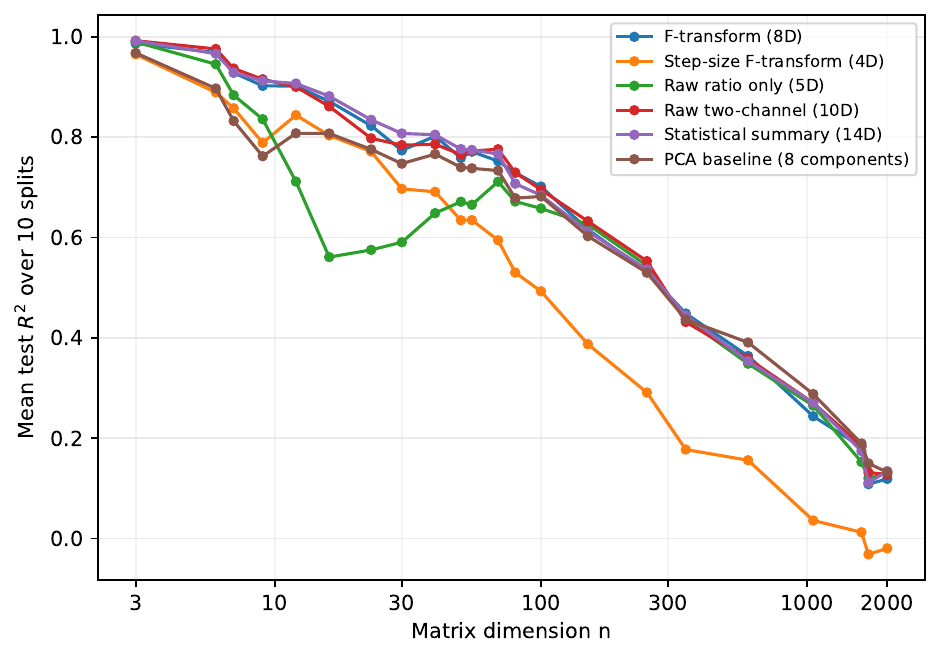}
\caption{\(R^2\) across matrix sizes.}
\label{fig:r2-vs-size}
\end{subfigure}
\hfill
\begin{subfigure}[t]{0.48\textwidth}
\centering
\includegraphics[width=\linewidth]
{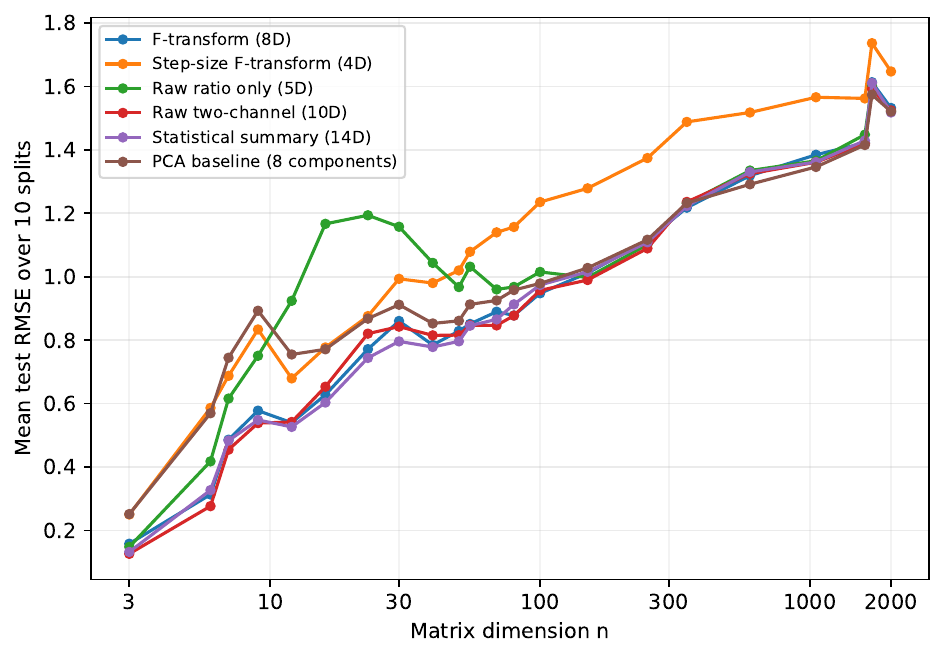}
\caption{RMSE across matrix sizes.}
\label{fig:rmse-vs-size}
\end{subfigure}

\caption{Predictive performance across matrix sizes for the principal
representation families. The horizontal axis is logarithmic to display
the tested dimensions clearly across the full range
\(3\le n\le2000\).}
\label{fig:performance-size}
\end{figure}
\subsection{Feature importance}
\label{sec:feature-importance}

Random Forest feature importance was examined to characterize how the
fitted predictive models distribute importance among the coordinates
of the different representation families. Figure~\ref{fig:feature-importance}
shows the mean impurity-based importance across the 22 matrix
dimensions for the full F-transform descriptor, the raw two-channel
representation, the statistical-summary representation, and the
step-size F-transform. For each dimension, importances were first
averaged over the ten train--test splits; the figure then summarizes
their variation across dimensions.

For the full eight-dimensional F-transform descriptor, the largest
mean importances are associated with \(u_3\) and \(v_3\), with mean
values of approximately \(0.340\) and \(0.336\), respectively.
The next largest contribution is \(v_2\), with mean importance
approximately \(0.102\). The remaining coordinates have smaller mean
importances, including \(r_2\) at approximately \(0.053\), while all
eight coordinates receive nonzero importance in the fitted models.

This pattern indicates that predictive information is distributed
across both channels rather than being concentrated exclusively in
the step-size coordinates. In particular, \(u_3\) represents the
right-side local level of the contraction-ratio channel, whereas
\(v_3\) represents the corresponding local level of the step-size
channel. Their prominence is consistent with the predictive relevance
of the later part of the observed early segment seen in the
observation-matched comparisons.

The numerical ranking should nevertheless be interpreted as a property
of the fitted Random Forest models rather than as an intrinsic ranking
of the mathematical coordinates. As shown in
Theorem~\ref{thm:rank-structure}, the linear map generating the eight
F-transform coordinates has rank six, and the coordinates satisfy two
exact linear dependencies. In the
presence of correlated or linearly dependent predictors, impurity-based
importance can be shared or redistributed among coordinates that carry
related information. The figure is therefore used to describe how the
predictive model exploits the representation, not to assign unique
independent contributions to individual F-transform coordinates.

The comparison across feature families in
Figure~\ref{fig:feature-importance} also shows that the importance
structure depends on the representation supplied to the model. Raw
samples, statistical summaries, and fuzzy-local coordinates expose the
same early trajectory information in different forms, and the fitted
Random Forest consequently distributes importance differently across
their respective inputs.

\begin{figure}[!htbp]
\centering
\includegraphics[width=0.90\textwidth]
{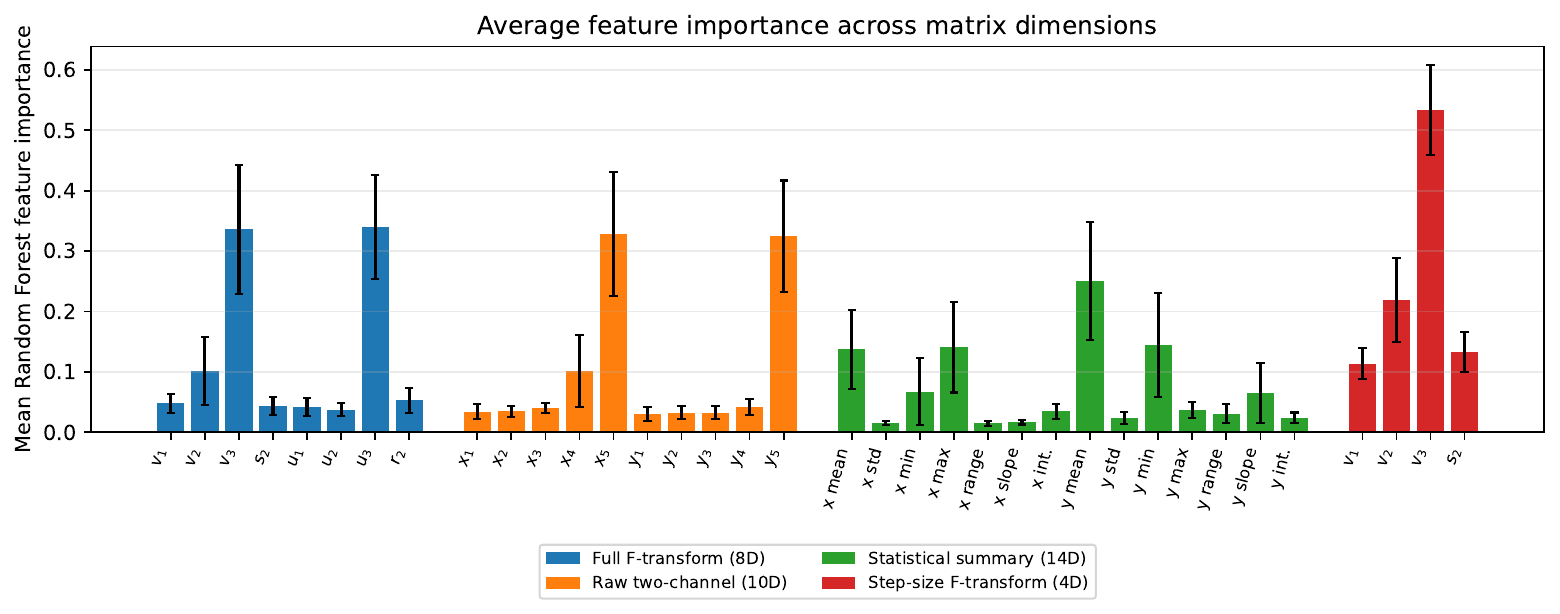}
\caption{Mean Random Forest feature importance across the 22 matrix
dimensions for selected representation families. Error bars show
95\% confidence intervals across dimension-level mean importances.
The importances are model-dependent and are interpreted within each
representation family.}
\label{fig:feature-importance}
\end{figure}

\subsection{Stability across repeated train--test partitions}
\label{sec:robustness}

The ten-split protocol used in the principal benchmark makes
train--test partition stability part of the main evaluation rather
than a separate secondary experiment. To quantify this variability,
the standard deviation of test \(R^2\) across the ten partitions was
computed for each representation within each matrix dimension, and
these within-dimension standard deviations were then averaged over the
22 tested dimensions.

For the full F-transform descriptor, the mean within-dimension
standard deviation of \(R^2\) is \(0.0351\). This is comparable to
the corresponding values for the statistical-summary representation
\((0.0331)\) and the raw two-channel representation \((0.0355)\).
The PCA-transformed raw representation has mean within-dimension
standard deviation \(0.0394\), while the step-size F-transform has
the largest variability among these principal representations,
with a value of \(0.0481\).

The observation-matched controls show a similar scale of variability.
The six-dimensional raw representation
\((q_2,\ldots,q_7)\) has mean within-dimension standard deviation
\(0.0418\), while the step-size F-transform augmented by \(q_7\)
has value \(0.0397\).

These results show that the predictive performance of the full
F-transform descriptor is not concentrated in a particular
train--test partition. Its split-to-split variability is of the same
order as that of the strongest larger baselines and is lower than that
of the step-size-only descriptor.

Because the same ten partitions are used for all feature families,
comparisons between representations are paired with respect to the
training and test data. The dimension-level comparisons reported in
Table~\ref{tab:paired} are based on the corresponding ten-split mean
\(R^2\) values and therefore summarize performance across repeated
partitions rather than a single data split.

\subsection{Sensitivity to representation design and observed segment}
\label{sec:sensitivity}

The baseline descriptor uses a five-point signal in each channel, a
three-node triangular fuzzy partition with bandwidth \(h=2\), three
zero-degree coefficients, and one centered first-degree coefficient.
Table~\ref{tab:sensitivity} examines how the predictive benchmark
changes when these choices are varied.

\begin{table}[ht]
\centering
\caption{Sensitivity of the convergence-length benchmark to fuzzy configuration and observed segment. Holm-adjusted \(p\)-values compare each alternative with the baseline three-node triangular \(L=5\) construction.}
\label{tab:sensitivity}
\scriptsize
\resizebox{0.98\textwidth}{!}{%
\begin{tabular}{lccccc}
\toprule
Configuration & Inputs & Observed segment & Mean \(R^2\) & Mean diff. & Holm \(p\) \\
\midrule
Baseline: triangular, 3 nodes, \(h=2\), \(L=5\) & 8 & \(q_2,\ldots,q_7\) & 0.649124 & -- & -- \\
1 node & 4 & \(q_2,\ldots,q_7\) & 0.638188 & -0.010936 & 0.2344 \\
5 nodes & 12 & \(q_2,\ldots,q_7\) & 0.652855 & 0.003731 & 0.4587 \\
Triangular, \(h=1.5\) & 8 & \(q_2,\ldots,q_7\) & 0.649124 & 0.000000 & 1.0000 \\
Triangular, \(h=2.5\) & 8 & \(q_2,\ldots,q_7\) & 0.649932 & 0.000808 & 1.0000 \\
Gaussian, 3 nodes & 8 & \(q_2,\ldots,q_7\) & 0.651506 & 0.002382 & 0.7479 \\
Boxcar, 3 nodes & 8 & \(q_2,\ldots,q_7\) & 0.649124 & 0.000000 & 1.0000 \\
Zero-degree levels only & 6 & \(q_2,\ldots,q_7\) & 0.641298 & -0.007826 & 0.0891 \\
Shorter segment, \(L=4\) & 8 & \(q_2,\ldots,q_6\) & 0.533623 & -0.115501 & $7.63\times 10^{-6}$ \\
Longer segment, \(L=6\) & 8 & \(q_2,\ldots,q_8\) & 0.745842 & 0.096718 & $4.29\times 10^{-6}$ \\
\bottomrule
\end{tabular}%
}
\end{table}

The baseline construction gives mean \(R^2=0.6491\). Variations that
retain the same observed segment \(q_2,\ldots,q_7\) produce relatively
small changes in predictive performance. Increasing the number of
fuzzy nodes from three to five gives mean \(R^2=0.6529\), while using
a single node gives \(0.6382\). On the five sampled positions, pointwise normalization makes the
triangular \(h=1.5\) and boxcar configurations coincide exactly with
the baseline sampled weights, so their identical benchmark values are
a discrete-grid identity rather than independent evidence of
robustness. The distinct triangular \(h=2.5\) and Gaussian
configurations also remain close to the baseline. None of these same-segment alternatives differs
significantly from the baseline after Holm adjustment.

Removing the first-degree trend coefficient has a somewhat larger
effect, reducing the mean \(R^2\) to \(0.6413\), although this
difference also does not remain significant after adjustment. Thus,
among the genuinely distinct same-segment alternatives tested, the
predictive result is not strongly tied to one precise choice of node
number, effective bandwidth, membership weighting, or retained
F-transform order.

A substantially larger effect appears when the amount of observed
trajectory information is changed. Shortening the construction to
\(L=4\), which uses the underlying segment
\[
q_2,\ldots,q_6,
\]
reduces the mean \(R^2\) to \(0.5336\), a mean difference of
\(-0.1155\) relative to the baseline. Extending the construction to
\(L=6\), which uses
\[
q_2,\ldots,q_8,
\]
increases the mean \(R^2\) to \(0.7458\), a mean difference of
\(0.0967\). Both changes remain significant after Holm adjustment.

The dominant sensitivity is therefore to the amount of early
trajectory information made available to the representation rather
than to the genuinely distinct nearby fuzzy-design alternatives tested
on the same observed segment. In particular,
the strong performance increase for the longer segment should not be
interpreted as evidence that the \(L=6\) fuzzy construction is
intrinsically superior to the baseline. It observes one additional
trajectory value and consequently has access to later dynamical
information. Conversely, the \(L=4\) representation is required to
operate from less information.

This distinction is consistent with the observation-matched results
in Section~\ref{sec:predictive-results}. The predictive benchmark is
sensitive to how far the early trajectory has been observed, while
different structured representations of the same observed segment
show considerably smaller differences.

\paragraph{Fuzzy-local and hard-local summaries.}

To examine the role of overlapping fuzzy weighting more directly, the
full F-transform descriptor was compared with the eight-dimensional
hard local-summary representation defined in
\eqref{eq:hard-local}. Both representations use the same underlying
segment \(q_2,\ldots,q_7\), both produce eight model inputs, and both
encode local levels together with a centered trend quantity.

\begin{table}[ht]
\centering
\caption{Equal-dimensional comparison between the proposed fuzzy-local coordinates and a hard local-summary construction on the same observed segment \(q_2,\ldots,q_7\).}
\label{tab:localbaseline}
\small
\begin{tabular}{lcccc}
\toprule
Representation & Inputs & Mean \(R^2\) & Mean diff. vs. F & Raw paired \(p\) \\
\midrule
Full F-transform & 8 & 0.649124 & -- & -- \\
Hard local summary & 8 & 0.646238 & -0.002886 & 0.1465 \\
\bottomrule
\end{tabular}
\end{table}

The hard local-summary representation gives mean \(R^2=0.6462\),
compared with \(0.6491\) for the F-transform descriptor. The mean
dimension-level difference, defined as hard local minus F-transform,
is
\[
-0.0029,
\]
with 95\% confidence interval
\[
[-0.0074,\;0.0017].
\]
The hard-local representation is higher in 7 dimensions and the
F-transform is higher in 15, while the paired Wilcoxon test gives
\(p=0.1465\).

The predictive benchmark therefore does not show a systematic
accuracy difference between the two local-summary constructions.
Much of the retained convergence information can be attributed to the
local level--trend organization itself. The F-transform provides this
organization through overlapping fuzzy supports and retains the
explicit fuzzy-local interpretation and stability properties
established in Section~\ref{sec:representation}, but the present
regression experiment does not indicate a separate predictive
advantage attributable solely to fuzzy weighting.
\subsection{Convergence-target diagnostics}
\label{sec:convergence-diagnostics}

The predictive experiments use the first \(10^{-6}\) crossing of the
step-size sequence as the convergence target,
\[
N_{\mathrm{conv}}
=
\min\{k\ge1:\delta_k<10^{-6}\}.
\]
To examine the stability of this definition, the complete set of
22,000 trajectories was evaluated for threshold recrossing, sustained
convergence, iteration-cap effects, nearby tolerance levels, and the
position of the descriptor's observed segment relative to the
convergence threshold.

\begin{table}[ht]
\centering
\caption{Convergence-target diagnostics over 22,000 trajectories.}
\label{tab:convdiag}
\small
\begin{tabularx}{0.92\textwidth}{X C{0.28\textwidth}}
\toprule
Diagnostic & Result \\
\midrule
Trajectories with a \(10^{-6}\) recrossing & 0/22000 \\
Total \(10^{-6}\) recrossing events & 0 \\
Trajectories reaching the 100-step operational cap & 0/22000 \\
Trajectories reaching the extended diagnostic cap & 0/22000 \\
Maximum last step required by the diagnostic computation & 39 \\
Observed segment \(q_2,\ldots,q_7\) strictly pre-threshold & 20842/22000 (94.74\%) \\
First crossing = three-consecutive crossing at \(10^{-5}\) & 22000/22000 \\
First crossing = three-consecutive crossing at \(10^{-6}\) & 22000/22000 \\
First crossing = three-consecutive crossing at \(10^{-7}\) & 22000/22000 \\
First crossing = three-consecutive crossing at \(10^{-8}\) & 22000/22000 \\
\bottomrule
\end{tabularx}
\end{table}

No trajectory returned above \(10^{-6}\) after its first crossing.
Moreover, the first threshold crossing coincided with the first
crossing sustained for three consecutive steps for all 22,000
trajectories. The same agreement was observed at
\[
10^{-5},\qquad 10^{-6},\qquad 10^{-7},\qquad 10^{-8}.
\]
Thus, within the generated trajectory ensemble, the first-crossing
definition and the stricter three-consecutive-step definition produce
the same convergence index.

The iteration limits do not affect the recorded target. None of the
22,000 trajectories reached either the original 100-step operational
limit or the extended diagnostic limit, and the largest step required
by the diagnostic computation was 39. Consequently, the observed
convergence-length distribution is not truncated by the iteration cap.

The complete underlying segment \(q_2,\ldots,q_7\) lies strictly
before the first \(10^{-6}\) crossing for
\[
20\,842\ \text{of}\ 22\,000
\]
trajectories, corresponding to \(94.74\%\) of the sample. For the
remaining 1,158 trajectories, the convergence threshold is reached
within or before the end of this observed segment. These cases are
retained because the descriptor is constructed from genuine computed
trajectory values and no post-threshold padding is used.

The convergence index is also stable under nearby tolerance choices.
The mean first-crossing iterations are approximately
\[
13.66,\quad 13.84,\quad 14.00,\quad 14.15
\]
for tolerances
\[
10^{-5},\quad10^{-6},\quad10^{-7},\quad10^{-8},
\]
respectively. Tightening the tolerance over three orders of magnitude
therefore changes the mean crossing time by less than one iteration.
Dimension-specific summaries of the convergence target, including
means, standard deviations, medians, quartiles, minima, and maxima,
are provided in the archived reproducibility outputs.

Together, these diagnostics show that the convergence-length target
used in the predictive experiments is not sensitive to immediate
threshold recrossing, the distinction between first and sustained
crossing, the iteration cap, or modest changes in the numerical
tolerance.
\subsection{Clustering structure and stability}
\label{sec:clustering}

The standardized F-transform descriptor exhibits reproducible coarse
organization under \(k\)-means clustering. Figure~\ref{fig:clustering}
summarizes cluster separation and initialization stability for
\[
k=2,\ldots,8.
\]

The mean silhouette coefficient is highest for \(k=2\), with value
\(0.353\), followed by \(k=3\) with \(0.288\). The mean silhouette
then decreases gradually as the number of clusters increases. The
\(k=2\) partition gives the highest silhouette in 20 of the 22 tested
matrix dimensions. Across the tested range, the observed mean
silhouette remains clearly above the corresponding permuted-feature
reference. In particular, for both \(k=2\) and \(k=3\), the observed
silhouette exceeds the 95th percentile of the corresponding
permuted-feature reference in all 22 dimensions.

The stability analysis gives a consistent picture. Repeated
\(k\)-means fits with different initialization seeds yield a mean
pairwise adjusted Rand index of approximately \(0.988\) for \(k=2\)
and \(0.970\) for \(k=3\). Stability decreases gradually for finer
partitions but remains high for smaller numbers of clusters. Thus, the
strongest evidence is for a stable two-group organization, while
\(k=3\) provides a reproducible finer partition of the same descriptor
space.

The association between clustering and convergence length was examined
separately by comparing the observed range of cluster mean
\(N_{\mathrm{conv}}\) values with a reference obtained by permuting the
convergence target while keeping the cluster assignments fixed. For
both \(k=2\) and \(k=3\), the observed separation exceeds the 95th
percentile of this reference in all 22 matrix dimensions. The
resulting coarse partitions therefore organize trajectories in a way
that is associated with measurable differences in convergence length.

The clustering is also stable under small perturbations of the
descriptor coordinates. For \(k=2\), adding independent Gaussian
noise in standardized coordinates with standard deviations
\(0.01\), \(0.05\), and \(0.10\) gives mean adjusted Rand indices
relative to the unperturbed partition of approximately
\(0.985\), \(0.961\), and \(0.942\), respectively. The corresponding
values for \(k=3\) are approximately \(0.977\), \(0.949\), and
\(0.900\). Thus, the strongest \(k=2\) partition remains highly
stable under the tested perturbations, while the reproducible
\(k=3\) subdivision shows the same gradual robustness pattern.

\begin{figure}[!htbp]
\centering
\includegraphics[width=0.95\textwidth]
{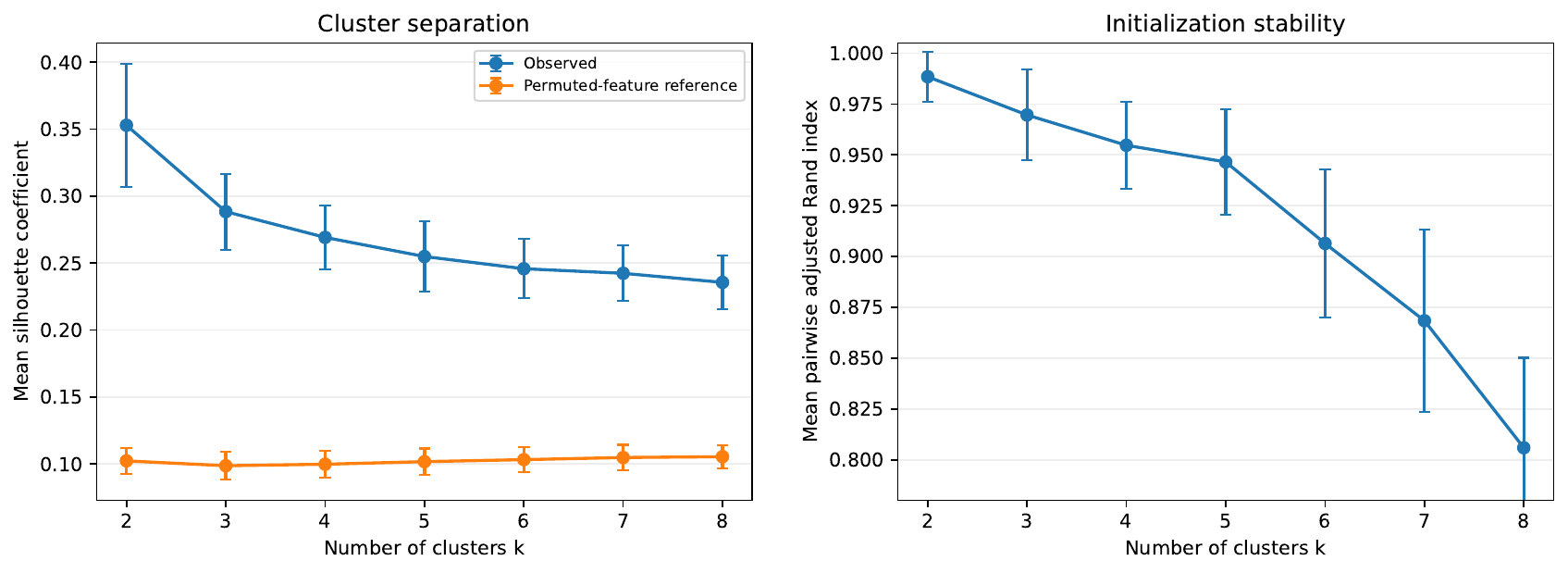}
\caption{Clustering structure of the standardized F-transform
descriptor across \(k=2,\ldots,8\). Left: mean silhouette coefficient
across the 22 matrix dimensions compared with a permuted-feature
reference. Right: mean pairwise adjusted Rand index across repeated
\(k\)-means initializations. Error bars show 95\% confidence intervals
across matrix dimensions.}
\label{fig:clustering}
\end{figure}

These results support a stable coarse organization of the descriptor
space, with the strongest separation at \(k=2\). They do not indicate
a unique discrete taxonomy of the dynamics; rather, the clustering
provides an exploratory description of how early contraction
trajectories are organized in the eight-dimensional representation.

\subsection{Low-dimensional concentration}
\label{sec:pca}

PCA of the standardized eight-dimensional F-transform descriptor
shows substantial concentration along a small number of dominant
directions. Across the 22 matrix dimensions, the first principal
component explains on average \(56.33\%\) of the variance and the
second explains \(28.14\%\). Together, the first two components
account for an average of
\[
84.47\%
\]
of the descriptor variance. The corresponding two-component total
ranges from \(78.43\%\) to \(90.09\%\) across the tested dimensions.

This concentration is consistent with the exact algebraic structure
established in Theorem~\ref{thm:rank-structure}. Although the
descriptor is represented by eight coordinates, these coordinates are
generated by six underlying logged step sizes and satisfy two exact
linear dependencies. Accordingly, the seventh and eighth principal
components carry only numerical-level variance. The PCA result goes
further than this exact rank reduction: within the six-dimensional
information space, most empirical variation is concentrated in the
first two principal directions.

For the representative dimension \(n=350\), the first principal
component explains \(55.16\%\) of the variance and the second explains
\(31.58\%\), for a combined total of \(86.73\%\).
Figure~\ref{fig:pca} shows the corresponding two-dimensional
projection, with trajectories grouped into empirical fast, medium,
and slow convergence-length bands.

\begin{figure}[!htbp]
\centering
\includegraphics[width=0.85\textwidth]
{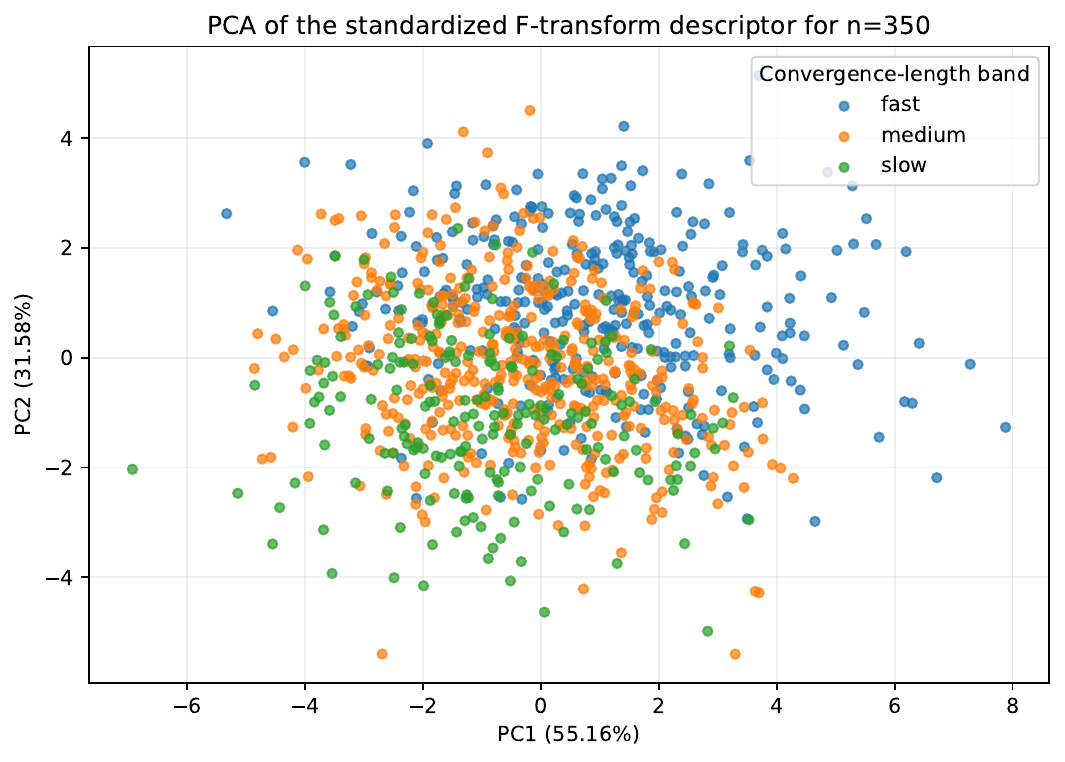}
\caption{PCA projection of the standardized F-transform descriptor
for \(n=350\), colored by empirical convergence-length band. The
first two principal components explain \(86.73\%\) of the variance
for this representative dimension.}
\label{fig:pca}
\end{figure}

The projection exhibits a broad ordering with respect to convergence
length rather than sharply separated classes. Faster and slower
trajectories tend to occupy different regions of the projected
descriptor space, while the medium-convergence trajectories overlap
with both. This organization is consistent with the clustering
results in Section~\ref{sec:clustering}, where the strongest evidence
is for a coarse partition rather than a unique discrete taxonomy.

Taken together, the PCA and clustering analyses indicate that the
descriptor is not merely a collection of unrelated coordinates. Its variation is strongly concentrated in a low-dimensional empirical
geometry, and that geometry is associated with differences in
convergence length.

\section{Discussion}\label{sec:discussion}

\subsection{What the descriptor contributes}

The proposed descriptor is best understood as a structured fuzzy
coordinate system for early nonlinear trajectory information. Its
main contribution is not a universal improvement in predictive
accuracy over raw or statistical representations. The empirical
results instead show that the descriptor retains nearly the same
convergence-related predictive information as larger feature families
while expressing the observed trajectory through explicit local
levels and trends in two dynamically meaningful channels.

The exact structure established in
Theorem~\ref{thm:rank-structure} sharpens this interpretation. The
eight coordinates satisfy
\[
\Psi=M(q_2,\ldots,q_7)^T,
\qquad
\operatorname{rank}M=6.
\]
Thus, the descriptor is an injective, overcomplete representation of
the six logged step sizes underlying the two channels. Passing from
\((q_2,\ldots,q_7)\) to \(\Psi\) therefore loses no information from
the observed segment. What changes is the coordinate organization:
the raw sequence is replaced by overlapping fuzzy-local levels and
centered trend coordinates for step size and contraction ratio.

This distinction is important when comparing the descriptor with the
raw and statistical baselines. Raw samples preserve the observed
trajectory values directly, but their local level and trend structure
remain implicit. Statistical summaries provide concise global
quantities but suppress temporal locality. The F-transform descriptor
occupies a different position: it preserves the information in the
six-point observed segment while reorganizing it into coordinates with
an explicit local dynamical interpretation.

The predictive results are consistent with this role. The full
descriptor achieves mean \(R^2=0.6491\), close to the raw
two-channel and statistical-summary representations, while the
observation-matched raw segment gives \(R^2=0.6421\). The paired
comparisons do not establish a systematic predictive advantage of the
F-transform over the matched raw coordinates. The empirical value of
the representation therefore lies in retaining comparable predictive
access while imposing a transparent fuzzy-local coordinate structure.

\subsection{Role of the contraction-ratio channel}

The contraction-ratio channel remains a central part of the
representation, but its role is more precise than an interpretation
based on independent additional information. Since
\[
\log\rho_k=q_{k+1}-q_k,
\]
the ratio channel is algebraically determined by the same logged
step-size sequence that underlies the complete descriptor. It does
not introduce new independent trajectory observations beyond
\(q_2,\ldots,q_7\).

Its value is instead representational. The step-size channel describes
the local magnitude and trend of the updates, whereas the
contraction-ratio channel makes the local evolution of that magnitude
explicit. The two channels therefore provide different coordinate
views of the same observed contraction segment.

The observation-matched experiments clarify this distinction. Adding
\(q_7\) to the step-size F-transform produces a large increase in
predictive performance, showing that access to the additional observed
step accounts for a substantial part of the difference between the
four-coordinate and full descriptors. Once the observed segment is
matched, the remaining difference between the full two-channel
descriptor and the \(q_7\)-augmented step-size descriptor is small.
The contraction-ratio coordinates should therefore be interpreted as
a structured representation of local contraction behavior rather than
as evidence of additional independent degrees of freedom.

The feature-importance analysis is consistent with this interpretation.
Coordinates from both channels are used by the fitted Random Forest
models, with \(u_3\) and \(v_3\) receiving the largest mean
importance under the ten-split protocol. Because the descriptor contains
exact linear dependencies, however, these importance values describe
the behavior of the fitted model rather than unique independent
contributions of individual coordinates.

\subsection{Fuzzy-local representation and design sensitivity}

The sensitivity experiments show that the baseline predictive result
is not unusually fragile to the genuinely distinct nearby
fuzzy-design alternatives tested. Changes in node number, effective
bandwidth, membership weighting, and retained F-transform order
produce only modest differences when the same underlying trajectory
segment is retained. Removing the centered first-degree term also changes
performance less than changing the amount of observed trajectory
information.

The stronger sensitivity to segment length is consistent with the
observation-matched benchmark. Extending the observed segment from
\(q_2,\ldots,q_7\) to \(q_2,\ldots,q_8\) produces a substantial
increase in predictive performance, whereas shortening it to
\(q_2,\ldots,q_6\) produces a substantial decrease. Predictive information therefore depends more strongly on how far the
early trajectory has been observed than on the genuinely distinct nearby
fuzzy-design alternatives tested on the same observed segment.

The comparison with the equal-dimensional hard local-summary
representation gives a complementary result. The hard-local
construction and the F-transform use the same observed segment, the
same number of model inputs, and a similar local level--trend
organization. The paired comparison does not detect a significant performance
difference between them in the present benchmark. This suggests that
much of the retained convergence information is associated with the
local organization itself rather than with fuzzy overlap alone.

The F-transform nevertheless provides additional structural
properties that are intrinsic to the representation rather than to
the downstream regression experiment. Its coordinates arise from an
explicit fuzzy partition, adjacent local summaries share information
through overlapping supports, the centered first-degree coefficient
recovers affine trends exactly, and the complete descriptor satisfies
the Lipschitz stability property established in
Proposition~\ref{prop:stability}. These properties distinguish the
construction from an arbitrary collection of local summary
statistics even when their predictive performance is similar.

\subsection{Why early trajectory coordinates are useful}

A fixed coordinate representation of an early trajectory is useful
when complete matrix trajectories are expensive to retain, when many
runs must be compared across dimensions or initial conditions, or
when geometric and statistical tools require a common finite-dimensional
representation.

In the present setting, the original states \(P_k\) are
\(n\times n\) matrices whose dimension changes with \(n\), whereas
the descriptor has a fixed eight-coordinate form for every tested
matrix size. This permits direct use of the same downstream tools
across dimensions and provides coordinates that can be interpreted in
terms of step-size level, step-size trend, contraction-ratio level,
and contraction-ratio trend.

For CONCOR-type and related repeated-correlation procedures, such
coordinates can support comparison, indexing, visualization, and
early characterization of convergence behavior without treating the
entire matrix trajectory as the representation object. The present
experiments use controlled random matrix ensembles rather than
empirical network datasets, so this operational interpretation is a
methodological motivation rather than a demonstrated application to
network blockmodeling.

The construction can also be viewed as a form of fuzzy information
granulation. The early contraction sequence is represented through
overlapping local summaries that preserve the observed information
while reorganizing it into interpretable granules of local level and
trend. In this sense, the descriptor connects the trajectory problem
with broader ideas in fuzzy information granulation
\cite{PedryczGranulation,ZadehGranulation}.

\subsection{Geometric interpretation}

The PCA and clustering analyses provide a geometric view complementary
to the predictive benchmark. Across the 22 matrix dimensions, the
first two principal components explain on average \(84.47\%\) of the
variance of the standardized descriptor. This concentration occurs
within the six-dimensional information space implied by the exact rank
of the descriptor and shows that the empirical variation is still more
concentrated than the algebraic rank alone requires.

The clustering analysis indicates a stable coarse organization of this
space. The strongest mean silhouette occurs at \(k=2\), while the
\(k=3\) partitions are also highly reproducible across initialization
seeds. Observed silhouette values remain well above the
permuted-feature reference over the tested range
\(k=2,\ldots,8\), and the cluster partitions are associated with
measurable differences in convergence length.

These results are better interpreted as evidence of structured
geometry than as a unique classification of trajectories into a fixed
number of dynamical regimes. The PCA projection similarly shows a
broad ordering of faster and slower trajectories rather than sharply
separated convergence classes. Prediction, PCA, and clustering thus
describe different aspects of the same representation: retained
information about convergence, concentration of variation, and coarse
organization of the trajectory coordinates.

\subsection{Convergence target and numerical reliability}

The convergence diagnostics support the use of the first
\(10^{-6}\) threshold crossing as an operational convergence length
for the trajectory ensemble studied here. No recrossings were observed
after the first threshold crossing, and the first crossing coincided
with the first three-consecutive-step crossing at all four tested
tolerances.

The target is also insensitive to the iteration cap: no trajectory
reached either the original 100-step limit or the extended diagnostic
limit. Tightening the threshold from \(10^{-5}\) to \(10^{-8}\)
changes the mean crossing time by less than one iteration.

Most observed segments are strictly pre-threshold:
\(94.74\%\) of the trajectories satisfy this condition for the full
segment \(q_2,\ldots,q_7\). The remaining trajectories are retained
because all descriptor coordinates are computed from genuine trajectory
values rather than padded post-convergence values. This distinction is
important for interpreting the representation as an early-trajectory
descriptor without introducing artificial constant tails into the
signals.

\subsection{Limitations}

The study is restricted to iterated Pearson correlation matrices under
one initialization law, with entries sampled independently from the
uniform distribution on \([-1,1]\). The empirical conclusions about
predictive performance, sensitivity, and geometric organization
therefore apply to this dynamical system and experimental ensemble.

The descriptor is constructed from a short early contraction segment.
Although the map from \(q_2,\ldots,q_7\) to the eight F-transform
coordinates is injective, information outside that observed segment is
not represented. The strong dependence of predictive performance on
segment length confirms that later observations can contain additional
information about the recorded convergence length.
Moreover, for 5.26\% of the tested trajectories, the \(10^{-6}\)
threshold is reached within or before the end of the full observed
segment. The benchmark is therefore not interpreted as a universally
pre-convergence online forecast.

The clustering results describe coarse empirical organization and
should not be interpreted as a proof of intrinsic dynamical classes.
Similarly, Random Forest feature importance is model-dependent,
particularly because the descriptor coordinates satisfy exact linear
dependencies.

The predictive comparisons are also conditional on the Random Forest
benchmark used in this study. Other downstream learners may exploit
the same representations differently, so the reported predictive
ranking should not be interpreted as model-independent.

The theoretical results established here concern the fixed
three-node, five-point construction. The rank characterization,
Lipschitz stability, and exact affine-trend recovery give a precise
mathematical description of this descriptor, but they do not by
themselves establish corresponding properties for arbitrary fuzzy
partitions, observation lengths, or other nonlinear dynamical systems.

\subsection{Future work}

Several extensions follow naturally from the present results. One
direction is to study how the rank and injectivity properties change
with the number of fuzzy nodes, observation length, and polynomial
degree of the F-transform. A second is to derive analytic relationships
between descriptor coordinates and known contraction properties of the
correlation map.

The empirical scope can also be extended to other initialization laws,
noisy trajectories, and empirical CONCOR or blockmodeling datasets.
Such experiments would determine which aspects of the representation
remain stable outside the controlled matrix ensemble studied here.

More generally, the same construction can be tested on other iterative
systems that admit meaningful step-size and contraction-ratio
observables. Such extensions should be treated as new empirical
questions rather than assumed from the present results. Comparison
with learned time-series representations and locality-preserving fuzzy
methods may also clarify the trade-off between analytically defined,
interpretable coordinates and data-adaptive latent representations.

\section{Reproducibility}
\label{sec:reproducibility}

The computational protocol is deterministic conditional on the stated
random seeds. For trajectory trial \(t\), matrix generation uses
\texttt{default\_rng(12345+t)} with independent
\(\mathcal U[-1,1]\) entries. The same Pearson correlation update,
Frobenius step definition, \(10^{-12}\) numerical floor, and
feature-construction formulas are used at every matrix dimension.

For the predictive experiments, Random Forest models use 300 trees
with estimator random state 42. Train--test partitions use seeds
\(0,\ldots,9\), and models are fitted separately for each matrix
dimension. The same partition is used across representation families
within each seed so that model comparisons are paired. PCA
preprocessing for the predictive baseline is fitted only on the
training portion of each split.

The archived computational materials include trajectory generation,
feature construction, repeated-split prediction, feature-importance
analysis, fuzzy-design sensitivity, convergence diagnostics, the
equal-dimensional hard local-summary comparison, PCA and clustering
analyses, and scripts that regenerate the numerical tables and figures
used in the manuscript from the stored outputs. The archive also records
the complete software environment and implementation settings used in
the experiments. The computational archive associated with this study
is deposited in Zenodo \cite{AlhajjHassan2026Code}.

\section{Conclusion}
\label{sec:conclusion}

This paper develops a two-channel F-transform representation for the
early contraction dynamics of iterated Pearson correlation matrices.
The construction starts from two observables with direct dynamical
meaning,
\[
\delta_k=\|P_{k+1}-P_k\|_F,
\qquad
\rho_k=\frac{\delta_{k+1}}{\delta_k},
\]
and represents their logarithmic step-size and contraction-ratio
signals through three fuzzy-local levels and one centered local trend
per channel.

The mathematical structure of the representation is exact. The eight
F-transform coordinates form a rank-six injective linear image of
\[
(q_2,\ldots,q_7),
\]
with two explicit linear dependencies. The observed six-value
log-step segment is therefore preserved exactly while being
reorganized into interpretable local coordinates for update magnitude
and contraction behavior. The construction also provides finite-window
Lipschitz stability and exact recovery of affine trends.

Across 22 matrix dimensions and 22,000 trajectories, the descriptor
retains convergence-length information close to the larger raw
two-channel and statistical-summary representations under a common
Random Forest benchmark. Observation-matched comparisons show that a
substantial part of the improvement over the step-size-only descriptor
comes from access to the additional trajectory value required by the
complete two-channel construction. The genuinely distinct nearby fuzzy-design alternatives tested have
comparatively small effects, whereas changing the amount of observed
trajectory information has a substantially larger effect.
An equal-dimensional hard local-summary representation gives similar
predictive performance, while the F-transform retains its systematic
overlapping fuzzy-local construction and exact analytic properties.

The representation also exhibits coherent empirical geometry. The
first two principal components explain on average \(84.47\%\) of the
descriptor variance across the tested dimensions, and clustering
reveals stable coarse organization, with the strongest mean silhouette
separation at \(k=2\). These geometric patterns are associated with
differences in convergence length.

The resulting contribution is a fixed-dimensional and interpretable
representation that is information-preserving with respect to the
observed six-value log-step segment. It provides a common coordinate system for
trajectory comparison, visualization, similarity analysis, and
downstream learning while retaining an explicit connection to the
underlying iterated-correlation dynamics.

\section*{Acknowledgments}

I thank Prof. Vil\'em Nov\'ak for introducing me to the theory of
fuzzy transforms and for valuable discussions. I am also grateful to
P. Zusmanovich for ongoing collaboration.

\section*{Data availability}

Software, data, computational outputs, and reproducibility materials
associated with this study are deposited in Zenodo
\cite{AlhajjHassan2026Code}.

\section*{Declaration of competing interests}

The author declares no known competing financial interests or personal
relationships that could have appeared to influence the work reported
in this paper.


\begin{thebibliography}{99}

\bibitem{AlhajjHassan2026Code}
Alhajj Hassan, I.
\newblock Software accompanying this study, Version 3.0.0 [software].
\newblock Zenodo, 2026.
\newblock doi:10.5281/zenodo.22009966.

\bibitem{EmpiricalLaws}
Alhajj Hassan, I.
\newblock Empirical Laws for Iterated Correlation Matrices.
\newblock arXiv:2512.15421, 2025.

\bibitem{FiniteBounds}
Alhajj Hassan, I.
\newblock Finite-step bounds for iterated correlation matrices.
\newblock \emph{Fixed Point Theory and Algorithms for Sciences and Engineering},
2026.
\newblock doi:10.1186/s13663-026-00844-6.

\bibitem{BreigerBoormanArabie1975}
Breiger, R.~L., Boorman, S.~A., Arabie, P.
\newblock An algorithm for clustering relational data with applications to
social network analysis and comparison with multidimensional scaling.
\newblock \emph{Journal of Mathematical Psychology} 12(3) (1975) 328--383.
\newblock doi:10.1016/0022-2496(75)90028-0.

\bibitem{Chen2002}
Chen, C.-H.
\newblock Generalized association plots: information visualization via
iteratively generated correlation matrices.
\newblock \emph{Statistica Sinica} 12(1) (2002) 7--29.

\bibitem{Fraikin2024}
Fraikin, A.~F., Bennetot, A., Allassonni\`ere, S.
\newblock T-Rep: Representation Learning for Time Series using Time-Embeddings.
\newblock \emph{International Conference on Learning Representations (ICLR)},
2024.

\bibitem{KreinovichPerfilieva2011}
Perfilieva, I., Kreinovich, V.
\newblock Fuzzy transforms of higher order approximate derivatives: A theorem.
\newblock \emph{Fuzzy Sets and Systems} 180(1) (2011) 55--68.
\newblock doi:10.1016/j.fss.2011.05.005.

\bibitem{KruskalCONCOR}
Kruskal, J.~B.
\newblock A theorem about CONCOR.
\newblock Technical Report MH-2-C--571,
Bell Laboratories, Murray Hill, NJ, 1978.

\bibitem{McQuittyClark1968}
McQuitty, L.~L., Clark, J.~A.
\newblock Clusters from iterative, intercolumnar correlational analysis.
\newblock \emph{Educational and Psychological Measurement}
28(2) (1968) 211--238.
\newblock doi:10.1177/001316446802800201.

\bibitem{NovakMirshahi2021}
Nov\'ak, V., Mirshahi, S.
\newblock On the similarity and dependence of time series.
\newblock \emph{Mathematics} 9(5) (2021) 550.
\newblock doi:10.3390/math9050550.

\bibitem{InsightFuzzyModeling}
Nov\'ak, V., Perfilieva, I., Dvo\v{r}\'ak, A.
\newblock \emph{Insight into Fuzzy Modeling}.
\newblock Wiley, Hoboken, NJ, 2016.

\bibitem{NovakPerfilievaHolcapekKreinovich2014}
Nov\'ak, V., Perfilieva, I., Hol\v{c}apek, M., Kreinovich, V.
\newblock Filtering out high frequencies in time series using F-transform.
\newblock \emph{Information Sciences} 274 (2014) 192--209.
\newblock doi:10.1016/j.ins.2014.02.133.

\bibitem{PedryczGranulation}
Pedrycz, W.
\newblock \emph{Granular Computing: Analysis and Design of Intelligent Systems}.
\newblock CRC Press, Boca Raton, FL, 2013.

\bibitem{Perfilieva2006}
Perfilieva, I.
\newblock Fuzzy transforms: theory and applications.
\newblock \emph{Fuzzy Sets and Systems} 157(8) (2006) 993--1023.
\newblock doi:10.1016/j.fss.2005.11.012.

\bibitem{PerfilievaDankovaBede2011}
Perfilieva, I., Da\v{n}kov\'a, M., Bede, B.
\newblock Towards a higher degree F-transform.
\newblock \emph{Fuzzy Sets and Systems} 180(1) (2011) 3--19.
\newblock doi:10.1016/j.fss.2010.11.002.

\bibitem{Yue2022}
Yue, Z., Wang, Y., Duan, J., Yang, T., Huang, C., Tong, Y., Xu, B.
\newblock TS2Vec: Towards universal representation of time series.
\newblock \emph{Proceedings of the AAAI Conference on Artificial Intelligence}
36(8) (2022) 8980--8987.
\newblock doi:10.1609/aaai.v36i8.20881.

\bibitem{ZadehGranulation}
Zadeh, L.~A.
\newblock Toward a theory of fuzzy information granulation and its centrality
in human reasoning and fuzzy logic.
\newblock \emph{Fuzzy Sets and Systems} 90(2) (1997) 111--127.
\newblock doi:10.1016/S0165-0114(97)00077-8.

\bibitem{Zhou2021}
Zhou, J., Pedrycz, W., Yue, X., Gao, C., Lai, Z., Wan, J.
\newblock Projected fuzzy C-means clustering with locality preservation.
\newblock \emph{Pattern Recognition} 113 (2021) 107748.
\newblock doi:10.1016/j.patcog.2020.107748.

\end{thebibliography}
\end{document}